\documentclass[12pt,leqno]{amsart}
\usepackage{amsmath,amsfonts,amssymb,amsthm}
\usepackage{graphicx}
\graphicspath{ }
\usepackage{wrapfig}
\usepackage{accents}
\usepackage{caption}
\usepackage{subcaption}
\usepackage{calligra}

\numberwithin{figure}{section}
\newcommand\norm[1]{\left\lVert#1\right\rVert}

\theoremstyle{plain}
\newtheorem{theorem}{Theorem}[section]

\newtheorem{prop}[theorem]{Proposition}

\theoremstyle{definition}
\newtheorem{defn}{Definition}[section]

\newtheorem{exmp}{Example}[section]

\theoremstyle{remark}
\newtheorem*{rem}{Remark}

\usepackage{mathtools}
\title{Some properties of geodesic $(\alpha,E)$-preinvex functions on a Riemannian manifold}
\author[A. A. Shaikh, C. K. Mondal, R. P. Agarwal]{Absos Ali Shaikh$^1$, Chandan Kumar Mondal$^2$, Ravi P Agarwal$^3$}

\address{\noindent\newline $^{1}$Department of Mathematics,\newline University of
Burdwan, Golapbag,\newline Burdwan-713104,\newline West Bengal, India}
\email{aask2003@yahoo.co.in, aashaikh@math.buruniv.ac.in}
\address{\noindent\newline $^{2}$School of sciences,\newline Netaji Subhas Open University,\newline Durgapur Campus, Durgapur-713214, \newline West Bengal, India}
\email{chan.alge@gmail.com, chandanmondal@wbnsou.ac.in}
\address{\noindent\newline $^{3}$Department of Mathematics,\newline Texas A\&M University-Kingsville, Kingsville,\newline Texas  78363-8202,\newline USA}
\email{Ravi.Agarwal@tamuk.edu}
\begin{document}
\begin{abstract}
In this article, we have introduced the concept of \textit{geodesic $(\alpha,E)$-invex set} and by using this concept the notion of \textit{geodesic $(\alpha,E)$-preinvex functions} and \textit{geodesic $(\alpha,E)$-invex functions} are developed on a Riemannian manifold. Moreover, several properties and results are deduced within aforesaid functions. An example is also constructed to illustrate the definition of geodesic $(\alpha,E)$-invex set. We have also established an important relation between geodesic $(\alpha,E)$-preinvex function and geodesic $(\alpha,E)$-invex function in a complete Riemannian manifold.
\end{abstract}
\noindent\footnotetext{ $^*$ Corresponding author.\\
$\mathbf{2010}$\hspace{5pt}Mathematics\; Subject\; Classification: 53C22, 58E10, 53B20.\\ 
{Key words and phrases: Geodesic $(\alpha,E)$-Invex sets, $(\alpha,E)$-Invex functions, Geodesic $(\alpha,E)$-Preinvex functions, Riemannian manifolds } }
\maketitle
\section{Introduction}
Convex sets and convex functions play an important and significant role in the theory of nonlinear programming and optimization. Since, the notion of convexity has great impact in real world problems, various authors are developing the new concept of convexity in order to extend the result to the larger class of optimization. Hanson  in 1981, made a significant step by introducing the concept of invexity. Hanson's work stimulates further development of the role and applications of invexity in mathematical programming and other branches of pure and applied mathematics. Later, the role of preinvex functions and invex sets in optimization theory, variational inequalities and equilibrium problems were studied by Jeyakumar \cite{JEY85}, Weir and Mond \cite{WEI88} and various authors. In 1999, Younees \cite{YOU99} generalized the concept of convexity by placing an operator $E$ in the same domain and named it as $E$-convex set and $E$-convex function. However, Yang \cite{YAN01} proved that some of the results given by Youness are not correct. Again the concept of strong $E$-convexity and semi-strong $E$-convexity was defined in \cite{YOU05}. \\

\indent 
Over the past few years, many results in the theory of nonlinear analysis and optimization on Riemannian manifolds have been extended from the Euclidean space. 
Convex functions in Riemannian manifolds, were studied by many authors see \cite{10,9,CK17,CK19,8}. Udri\c{s}te \cite{UDR94} and Rapcsak \cite{1} developed the concept of geodesic convexity in Riemannian manifold. Pini \cite{PI94} introduced the concept of Riemannian invexity, while Mititelu \cite{MI01} investigated its generalization.  In 2012, Iqbal et al. \cite{3} established the notion of geodesic E-convex set and functions and also they have discussed their properties and results. Barani et al. \cite{BAR07} introduced the concepts of geodesic invex set and geodesic preinvex functions on Riemannian manifolds with respect to particular maps. In 2012, Agarwal et al. \cite{AM12} generalized the notion of invexity and developed the concept of $\alpha$-invex sets and $\alpha$-preinvex functions in Riemannian manifold. Agian, Kumari and Jayswal \cite{KJ18} introduced the notion of geodesic E-preinvex function and geodesic semi E-preinvex function on Riemannian manifold and investigated some of its properties.  Motivating by the work of Agarwal et al. and Kumari and Jayswal, we have developed the concept of $(\alpha,E)$-invex sets and $(\alpha,E)$-preinvex functions in a complete Riemannian manifold and investigated some of its properties.\\

\indent 
The paper is structured as follows: Section 2 deals with some well known facts of Riemannian manifolds, geodesic convexity and geodesic invexity. In the next section, we have defined \textit{geodesic $(\alpha,E)$-invex set} in a complete Riemannian manifold. And by using this definition we have developed the concept of \textit{geodesic $(\alpha,E)$-preinvex} and \textit{geodesic $(\alpha,E)$-invex functions}. In the last section, we have deduced a relation between geodesic $(\alpha,E)$-preinvex and geodesic $(\alpha,E)$-invex function,( see Theorem 4.2), which is the main result of this paper. Also we have deduced some properties of $(\alpha,E)$-preinvex and geodesic $(\alpha,E)$-invex functions.

\section{Notations and Preliminaries}
In this section we have recalled some basic concept of a Riemannian manifold $(M,g)$, which is necessary throughout this paper (for reference see \cite{JOS11}). The length $L(\gamma)$ of the curve $\gamma:[a,b]\rightarrow M$ is given by
\begin{eqnarray*}
L(\gamma)&=&\int_{a}^{b}\sqrt{g_{\gamma(t)}(\dot{\gamma}(t),\dot{\gamma}(t))}\ dt\\
&=&\int_{a}^{b}\norm {\dot{\gamma}(t)}dt.
\end{eqnarray*}
The curve $\gamma$ is said to be a geodesic if its velocity vector $\gamma'(t)$ is parallel along $\gamma(t)$, i.e., $\nabla_{\dot{\gamma}(t)}\dot{\gamma}(t)=0\ \forall t\in [a,b]$, where $\nabla$ is the Riemannian connection of $g$.
For any point $p\in M$ and the unique geodesic $\gamma_u$ with $\gamma(0)=p$ and $\dot{\gamma}_u(0)=u$, the exponential map $exp_p:V_p\rightarrow M$ is defined by
$$exp_p(u)=\gamma_u(1),$$
where  and $V_p$ is a collection of vectors of $T_pM$ such that for each element $u\in V_p$, the geodesic with initial tangent vector $u$ is defined on $[0,1]$. It can be easily seen that the norm of a tangent vector is constant for a geodesic $\gamma$. A smooth vector field is a smooth function $X:M\rightarrow TM$ such that $\pi\circ X=id_M$, where $\pi:TM\rightarrow M$ is the projection map, i.e, in each point of the manifold $M$ we smoothly choose a tangent vector. The gradient of a function $f\in C^\infty(M)$ at the point $p\in M$ is defined by $\nabla f(p)=g^{ij}\frac{\partial f}{\partial x_{j}}\frac{\partial}{\partial x_i}\mid_p,$ where $\{\frac{\partial}{\partial x_i}\}_{i=1}^n|_p$ is an orthonormal coordinate system for $T_pM$.
\begin{defn}\cite{UDR94}
 A real valued function $f$ on $M$ is called convex if 
\begin{equation*}
f\circ\gamma(t)\leq (1-t)f\circ\gamma(0)+tf\circ\gamma(1)\quad \forall t\in [0,1],
\end{equation*}
for every geodesic $\gamma:[0,1]\rightarrow M$.
\end{defn}
\begin{defn}\cite{BAR07}
Let $M$ be a Riemannian manifold and $\eta: M\times M\rightarrow TM$ be a function such that for every $x,y\in M$, $\eta(x,y)\in T_yM$. A nonempty subset $S$ of $M$ is said to be geodesic invex with respect to $\eta$ if for every $x,y\in S$ there exists exactly one geodesic $\gamma_{x,y}:[0,1]\rightarrow M$ such that
$$\gamma_{x,y}(0)=y,\ \gamma'_{x,y}(0)=\eta(x,y),\ \gamma_{x,y}(t)\in S, \quad \forall t\in [0,1].$$
\end{defn}
\begin{defn}\cite{BAR07}
Let $S$ be an open set of the Riemannian manifold $M$ which is geodesic invex set with respect to $\eta:M\times M\rightarrow TM$. A differentiable function $f:S\rightarrow \mathbb{R}$ is said to be $\eta$-invex on $S$ if the following condition holds
$$f(x)-f(y)\geq df_y(\eta(x,y)),\ \forall x,y\in S.$$
$f$ is called geodesic $\eta$-preinvex if for every $x,y\in S$,
$$f(\gamma_{x,y}(t))\leq tf(x)+(1-t)f(y),\ \forall t\in [0,1].$$
\end{defn}
\section{$(\alpha,E)$-invex set and $(\alpha,E)$-preinvex function}
\indent Suppose $E:M\rightarrow M$ is a function and $\alpha:M\times M\rightarrow\mathbb{R}-\{0\}$ is a bifunction. We denote the complete Riemannian manifold by the notation $M$.
\begin{defn}
Let $\eta:M\times M\rightarrow TM$ be a function and $\alpha:M\times M\rightarrow\mathbb{R}-\{0\}$ be a bifunction such that $\alpha(x,y)\eta(x,y)\in T_yM$ for every $x,y\in M$. A subset $S\neq\phi$ of $M$ is said to be a geodesic $(\alpha,E)$-invex set with respect to $\eta$ and $\alpha$ if for every $x,y\in S$, there exists exactly one geodesic $\gamma_{E(x),E(y)}:[0,1]\rightarrow M$ such that
$$\gamma_{E(x),E(y)}(0)=E(y),\ \gamma'_{E(x),E(y)}(0)=\alpha(E(x),E(y))\eta(E(x),E(y)),\ \gamma_{E(x),E(y)}(t)\in S, \quad \forall t\in [0,1].$$
\end{defn}
If $\alpha(p,q)=1$ for all $p,q\in M$ and $E$ is the identity map, then geodesic $(\alpha,E)-$invex set becomes geodesic invex set \cite{BAR07}. Geodesic  $(\alpha,E)-$invex set reduces to geodesic $\alpha$-invex set \cite{AM12} if $E(p)=p$ for all $p\in M$. If  only $\alpha \equiv 1$, then geodesic $(\alpha,E)-$invex set becomes geodesic $E$-invex set \cite{KJ18}.
\begin{defn}
Let $S$ be an open subset of $M$ which is geodesic $(\alpha,E)$-invex set with respect to $\eta:M\times M\rightarrow TM$ and $\alpha:M\times M\rightarrow\mathbb{R}-\{0\}$. A function $f:S\rightarrow\mathbb{R}$ is said to be geodesic $(\alpha,E)-$preinvex if 
$$f(\gamma_{E(x),E(y)}(t))\leq tf(E(x))+(1-t)f(E(y)), \quad \forall x,y\in S\text{ and }t\in [0,1],$$
where $\gamma_{E(x),E(y)}$ is the unique geodesic.
\end{defn}
 
\begin{defn}
Let $S$ be an open subset of $M$ which is geodesic $(\alpha,E)$-invex set with respect to $\eta:M\times M\rightarrow TM$ and $\alpha:M\times M\rightarrow\mathbb{R}-\{0\}$. A differentiable function $f:S\rightarrow\mathbb{R}$ is said to be geodesic $(\alpha,E)-$invex if 
$$f(E(x))-f(E(y))\geq df_{E(y)}(\alpha(E(x),E(y))\eta(E(x),E(y))), \quad \forall x,y\in S.$$
\end{defn}
\begin{exmp}
Let $M$ be a Cartan-Hadamard manifold and $x_0,y_0\in M$ such that $x_0\neq y_0$. Consider two open ball $B(x_0,r_1)$ and $B(y_0,r_2)$ of radius $r_1$ and $r_2$ respectively, such that $B(x_0,r_1)\cap B(y_0,r_2)=\phi $ for some $0<r_1,r_2<\frac{1}{2}d(x_0,y_0)$. Now take
$$A=B(x_0,r_1)\cup B(y_0,r_2).$$ Then it is obvious that $A$ is not geodesic convex. Now define the functions $E:M\rightarrow M$ and $\eta:M\times M\rightarrow M$ such that

\[\eta(x,y)= \begin{cases} 
      exp_y^{-1}x & x,y\in B(x_0,r_1)\text{ or }x,y\in B(y_0,r_2) \\
      0_y & \text{ otherwise.} 
   \end{cases}
\]
\begin{equation*}
E(x)=\{y\in \gamma_{x,y}:d(x_0,y)=r_1/2\} \text{ for all }x\in A,
\end{equation*}
where $\gamma_{x,y}$ denotes the geodesic joining $x$ and $y$ whose existence is ensured in Cartan-Hadamard manifolds.\\
For every $x,y\in M$, consider a bifunction $\alpha:M\times M\rightarrow\mathbb{R}-\{0\}$ and $\gamma_{E(x),E(y)}:[0,1]\rightarrow M$ defined by
$$\gamma_{E(x),E(y)}(t)=exp_{E(y)}(t\alpha(E(x),E(y))\eta(E(x),E(y))), \text{ for all }t\in [0,1].$$
Then, 
$$\gamma_{E(x),E(y)}(0)=E(y),\ \gamma'_{E(x),E(y)}(0)=\alpha(E(x),E(y))\eta(E(x),E(y)).$$
Now simple calculation shows that $A$ is a geodesic $(\alpha,E)$-invex set.
\end{exmp}
\section{Main Results}

\begin{prop}
Suppose $S$ is an open subset of $M$ which is geodesic $(\alpha,E)$-invex set with respect to $\eta:M\times M\rightarrow TM$ and $\alpha:M\times M\rightarrow\mathbb{R}-\{0\}$. Suppose $f:S\rightarrow\mathbb{R}$ is  a geodesic $(\alpha,E)$-preinvex function, then every lower section of $f$ defined by 
$$S_\lambda:=\{x\in S:f(x)\leq \lambda\},\quad \lambda\in \mathbb{R},$$
is a geodesic $(\alpha,E)$-invex set with respect to $\eta$ and $\alpha$.
\end{prop}
\begin{proof}
Let $x,y\in S_\lambda$. Since $S$ is a geodesic $(\alpha,E)$-invex set with respect to $\eta$ and $\alpha$, there exists exactly one geodesic $\gamma_{E(x),E(y)}:[0,1]\rightarrow M$ such that 
$$\gamma_{E(x),E(y)}(0)=E(y),\ \gamma'_{E(x),E(y)}(0)=\alpha(E(x),E(y))\eta(E(x),E(y)),\ \gamma_{E(x),E(y)}(t)\in S, \quad \forall t\in [0,1].$$
Now by geodesic $(\alpha,E)$-preinvexity of $f$ we have
\begin{eqnarray*}
f(\gamma_{E(x),E(y)}(t))&\leq& tf(E(x))+(1-t)f(E(y))\\
&\leq& t\lambda+(1-t)\lambda=\lambda.
\end{eqnarray*}
Hence $\gamma_{E(x),E(y)}(t)\in S_\lambda\ \forall t\in [0,1]$. Therefore, $S_\lambda$ is a geodesic $(\alpha,E)$-invex set with respect to $\eta$ and $\alpha$.
\end{proof}
\begin{theorem}
Let $S\subset M$ be a geodesic $(\alpha,E)$-invex set with respect to $\eta$ and $\alpha$. If a function $f:S\rightarrow \mathbb{R}$ is differentiable and geodesic $(\alpha,E)$-preinvex on $S$, then $f$ is a geodesic $(\alpha,E)$-invex function on $S$.
\end{theorem}
\begin{proof}
Since $S$ is a geodesic $(\alpha,E)$-invex set with respect to $\eta$ and $\alpha$, there exists exactly one geodesic $\gamma_{E(x),E(y)}:[0,1]\rightarrow M$ such that 
$$\gamma_{E(x),E(y)}(0)=E(y),\ \gamma'_{E(x),E(y)}(0)=\alpha(E(x),E(y))\eta(E(x),E(y)),\ \gamma_{E(x),E(y)}(t)\in S, \quad \forall t\in [0,1].$$
Again $f$ is geodesic $(\alpha,E)$-preinvex function. therefore, we have
$$f(\gamma_{E(x),E(y)}(t))\leq tf(E(x))+(1-t)f(E(y)),$$
i.e.,
$$f(\gamma_{E(x),E(y)}(t))-f(E(y))\leq t(f(E(x))-f(E(y))). $$
On dividing by $t$, we get
$$\frac{1}{t}[f(\gamma_{E(x),E(y)}(t))-f(E(y))]\leq (f(E(x))-f(E(y))). $$
Now taking limit as $t\rightarrow 0$
$$df_{\gamma_{E(x),E(y)}(0)}(\gamma'_{E(x),E(y)}(0))\leq f(E(x))-f(E(y)).$$
Therefore,
$$ df_{E(y)}(\alpha(E(x),E(y))\eta(E(x),E(y)))\leq f(E(x))-f(E(y)).$$
Which shows that $f$ is a geodesic $(\alpha,E)$-invex function.
\end{proof}
\begin{defn}(Property (P))
Let $M$ be a Riemannian manifold and $\gamma_{E(x),E(y)}:[0,1]\rightarrow M$ be a curve on $M$ such that $\gamma_{E(x),E(y)}(0)=E(y)$ and $\gamma_{E(x),E(y)}(1)=E(x)$. Then $\gamma_{E(x),E(y)}$ is said to possess the Property (P) with respect to $x,y\in M$ if
$$\gamma'_{E(x),E(y)}(s)(t-s)=\alpha(\gamma_{E(x),E(y)}(t),\gamma_{E(x),E(y)}(s))\eta(\gamma_{E(x),E(y)}(t),\gamma_{E(x),E(y)}(s)),\ \forall s,t\in [0,1].$$
\end{defn}
\begin{rem}
If $\alpha=1$, then the above property reduces to the property defined by Kumari and Jayswal \cite{KJ18}. Agarwal et. al. \cite{AM12} defined the above property when $E$ is the identity map. If $\alpha=1$ and $E$ is the identity map, then the above property is defined by Pini \cite{PI94}.
\end{rem}
Let $M$ be a Riemannian manifold and $\gamma_{E(x),E(y)}$ possessing the Property (P) with respect to $x,y\in M$, then
\begin{eqnarray*}
\alpha(E(x),E(y))\eta(E(x),E(y))&=&\alpha(\gamma_{E(x),E(y)}(1),\gamma_{E(x),E(y)}(0))\eta(\gamma_{E(x),E(y)}(1),\gamma_{E(x),E(y)}(0))\\
&=&\gamma'_{E(x),E(y)}(0).
\end{eqnarray*}
In this case where $\gamma_{E(x),E(y)}$ is a geodesic, then
\begin{eqnarray*}
&&\alpha(\gamma_{E(x),E(y)}(0),\gamma_{E(x),E(y)}(s))\eta(\gamma_{E(x),E(y)}(0),\gamma_{E(x),E(y)}(s))\\
&&= -s\gamma'_{E(x),E(y)}(s)\\
&&= -sP^s_{0,\gamma_{E(x),E(y)}}[\gamma'_{E(x),E(y)}(0)]\\
&&= -sP^s_{0,\gamma_{E(x),E(y)}}[\alpha(E(x),E(y))\eta(E(x),E(y))].
\end{eqnarray*}
or,
$$P^0_{s,\gamma_{E(x),E(y)}}[\alpha(E(y),\gamma_{E(x),E(y)}(s))\eta(E(y),\gamma_{E(x),E(y)}(s))]=-s\alpha(E(x),E(y))\eta(E(x),E(y)).$$
\begin{eqnarray*}
&&\alpha(\gamma_{E(x),E(y)}(1),\gamma_{E(x),E(y)}(s))\eta(\gamma_{E(x),E(y)}(1),\gamma_{E(x),E(y)}(s))\\
&&= (1-s)\gamma'_{E(x),E(y)}(s)\\
&&= (1-s)P^s_{0,\gamma_{E(x),E(y)}}[\gamma'_{E(x),E(y)}(0)]\\
&&= (1-s)P^s_{0,\gamma_{E(x),E(y)}}[\alpha(E(x),E(y))\eta(E(x),E(y))].
\end{eqnarray*}
or,
$$P^0_{s,\gamma_{E(x),E(y)}}[\alpha(E(x),\gamma_{E(x),E(y)}(s))\eta(E(x),\gamma_{E(x),E(y)}(s))]=(1-s)\alpha(E(x),E(y))\eta(E(x),E(y)).$$
Therefore,
\begin{equation*}
  (C) 
    \begin{cases}
      P^0_{s,\gamma_{E(x),E(y)}}[\alpha(E(y),\gamma_{E(x),E(y)}(s))\eta(E(y),\gamma_{E(x),E(y)}(s))]&=-s\alpha(E(x),E(y))\eta(E(x),E(y)\\
      P^0_{s,\gamma_{E(x),E(y)}}[\alpha(E(x),\gamma_{E(x),E(y)}(s))\eta(E(x),\gamma_{E(x),E(y)}(s))]&=(1-s)\alpha(E(x),E(y))\\
      &\eta(E(x),E(y)),
    \end{cases}       
\end{equation*}
for all $s\in [0,1]$. The above two conditions we call Condition (C).
\begin{theorem}
Let $S$ be an open subset of $M$ which is geodesic $(\alpha,E)$-invex set with respect to $\eta:M\times M\rightarrow TM$ and $\alpha:M\times M\rightarrow\mathbb{R}-\{0\}$. Let $f:S\rightarrow\mathbb{R}$ be a differentiable function and satisfies the condition (C). Then $f$ is a geodesic $(\alpha,E)$-preinvex on $S$ if $f$ is $(\alpha,E)$-invex on $S$.
\end{theorem}
\begin{proof}
Since $S$ is a geodesic $(\alpha,E)$-invex set with respect to $\eta$ and $\alpha$, there exists a unique geodesic  $\gamma_{E(x),E(y)}:[0,1]\rightarrow M$ such that 
$$\gamma_{E(x),E(y)}(0)=E(y),\ \gamma'_{E(x),E(y)}(0)=\alpha(E(x),E(y))\eta(E(x),E(y)),\ \gamma_{E(x),E(y)}(t)\in S, \quad \forall t\in [0,1].$$
Now, fix $t\in [0,1]$ and set $q=\gamma_{E(x),E(y)}(t)$, then using $(\alpha,E)$-invexity of $f$ on $S$, we have
\begin{equation}\label{eq1}
f(E(x))-f(E(q))\geq df_{E(q)}(\alpha(E(x),E(q))\eta(E(x),E(q))).
\end{equation}
\begin{equation}\label{eq2}
f(E(y))-f(E(q))\geq df_{E(q)}(\alpha(E(y),E(q))\eta(E(y),E(q))).
\end{equation}
On multiplying (\ref{eq1}) by $t$ and (\ref{eq2}) by $(1-t)$, respectively, and then adding we get
\begin{eqnarray}\label{eq3}
tf(E(x))+(1-t)f(E(y))-f(E(q))&\geq & df_{E(q)}[t\alpha(E(x),E(q))\eta(E(x),E(q))\\ \nonumber &&+(1-t)\alpha(E(y),E(q))\eta(E(y),E(q))].
\end{eqnarray}
By the condition (C), we have
\begin{eqnarray*}
&& t\alpha(E(x),E(q))\eta(E(x),E(q))+(1-t)\alpha(E(y),E(q))\eta(E(y),E(q))\\
&&= t(1-t)P^t_{0,\gamma}[\alpha(E(x),E(y))\eta(E(x),E(y))]-(1-t)tP^t_{0,\gamma}[\alpha(E(x),E(y))\eta(E(x),E(y))]\\
&&=0.
\end{eqnarray*}
Now combining the above relation with the inequality (\ref{eq3}), we get
$$tf(E(x))+(1-t)f(E(y))-f(E(q))\geq 0.$$
Which implies that
$$f(\gamma_{E(x),E(y)}(t))\leq tf(E(x))+(1-t)f(E(y)).$$
Therefore, the function $f$ is a geodesic $(\alpha,E)$-preinvex on $S$.
\end{proof}
\begin{theorem}
Let $S\subset M$ be a geodesic $(\alpha,E)$-invex set with respect to $\eta:M\times M\rightarrow TM$ and $\alpha:M\times M\rightarrow R$. 
A function $f:S\rightarrow R$ is a geodesic $(\alpha,E)$-preinvex and $\phi:I\rightarrow R$ is an increasing geodesic pre-invex such that range $f\subset I$, where $I$ is any interval. Then the composite function $\phi\circ f$ is a geodesic $(\alpha,E)$-preinvex on S.
\end{theorem}
\begin{proof}
Since $S$ is a geodesic $(\alpha,E)$-invex set with respect to $\eta$ and $\alpha$, there exists a unique geodesic  $\gamma_{E(x),E(y)}:[0,1]\rightarrow M$ such that 
$$\gamma_{E(x),E(y)}(0)=E(y),\ \gamma'_{E(x),E(y)}(0)=\alpha(E(x),E(y))\eta(E(x),E(y)),\ \gamma_{E(x),E(y)}(t)\in S, \quad \forall t\in [0,1].$$
By the defnition of geodesic $(\alpha,E)$-preinvex function, we have
$$f(\gamma_{E(x),E(y)}(t))\leq tf(E(x))+(1-t)f(E(y)).$$
As $\phi$ is an increasing geodesic pre-invex function, we get
$$\phi\circ f(\gamma_{E(x),E(y)}(t))\leq \phi (tf(E(x))+(1-t)f(E(y)))$$
\begin{eqnarray*}
(\phi\circ f)(\gamma_{E(x),E(y)}(t))&\leq & t\phi (f(E(x)))+(1-t)\phi(f(E(y)))\\
&=&  t(\phi\circ f)(E(x))+(1-t)(\phi\circ f)(E(y)).
\end{eqnarray*}
Therefore, $\phi\circ f$ is a geodesic $(\alpha,E)$-preinvex function on S.
\end{proof}
\begin{theorem}
Let $S\subset M$ be a geodesic $(\alpha,E)$-invex set with respect to $\eta:M\times M\rightarrow TM$ and $\alpha:M\times M\rightarrow R$. If $f_i:S\rightarrow R$, $j\in J$ are geodesic $(\alpha,E)$-preinvex functions on S such that $\sup_{j\in J}f_j$ exist in $\mathbb{R}$, then the function $f:S\rightarrow R$ defined by 
$$f(x)=\sup_{j\in J}f_j(x), \text{ for }x\in S,$$
is a geodesic $(\alpha,E)$-preinvex function on S.
\end{theorem}
\begin{proof}
Since $S$ is a geodesic $(\alpha,E)$-invex set with respect to $\eta$ and $\alpha$, there exists a unique geodesic  $\gamma_{E(x),E(y)}:[0,1]\rightarrow M$ such that 
$$\gamma_{E(x),E(y)}(0)=E(y),\ \gamma'_{E(x),E(y)}(0)=\alpha(E(x),E(y))\eta(E(x),E(y)),\ \gamma_{E(x),E(y)}(t)\in S, \quad \forall t\in [0,1].$$
From the geodesic $(\alpha,E)$-preinvexity of $f_j$, $j\in J$, we have
$$f_j(\gamma_{E(x),E(y)}(t))\leq tf_j(E(x))+(1-t)f_j(E(y)).$$
Then,
\begin{eqnarray*}
\sup_{j\in J}f_j(\gamma_{E(x),E(y)}(t))&\leq&\sup_{j\in J}\Big( tf_j(E(x))+(1-t)f_j(E(y))\Big)\\
&\leq & t\Big(\sup_{j\in J} f_j(E(x))\Big)+(1-t)\Big(\sup_{j\in J}f_j(E(y))\Big)\\
&=&tf(E(x))+(1-t)f(E(y)).
\end{eqnarray*}
Therefore, we get
$$f(\gamma_{E(x),E(y)}(t))\leq tf(E(x))+(1-t)f(E(y)).$$ This proves the $(\alpha,E)$-preinvexity of $f$.
\end{proof}
\begin{theorem}
Let $S$ be a geodesic $(\alpha,E)$-invex set with respect to $\eta:M\times M\rightarrow TM$ and $\alpha:M\times M\rightarrow\mathbb{R}$. Suppose $F:S\times S\rightarrow\mathbb{R}$ is a continuous geodesic $(\alpha,E)$-preinvex, i.e., $F$ is geodesic $(\alpha,E)$-preinvex with respect to each variable. Then the function $f:S\rightarrow\mathbb{R}$ defined by
$$f(p)=\inf_{q\in S}F(p,q),$$
is a geodesic $(\alpha,E)$-preinvex function on $K$.
\end{theorem}
\begin{proof}
Suppose $\epsilon>0$ is an arbitrary small number and $p_0,p_1\in S$. Since $S$ is a geodesic $(\alpha,E)$-invex set with respect to $\eta$ and $\alpha$, there exists a unique geodesic  $\gamma_{E(p_0),E(p_1)}:[0,1]\rightarrow M$ such that, for all $t\in [0,1],$
$$\gamma_{E(p_0),E(p_1)}(0)=E(p_1),\ \gamma'_{E(p_0),E(p_1)}(0)=\alpha(E(p_0),E(p_1))\eta(E(p_0),E(p_1)),\ \gamma_{E(p_0),E(p_1)}(t)\in S.$$
Now from the definition of $f$, we get, there exists $q_0,q_1\in S$ such that
$$F(p_1,q_1)<f(p_1)+\epsilon,\ F(p_0,q_0)<f(p_0)+\epsilon.$$
By the geodesic $(\alpha,E)$-set with respect to $\eta$ and $\alpha$, there exists exactly one geodesic $\lambda_{E(q_0),E(q_1)}:[0,1]\rightarrow M$ such that, for all $s\in [0,1],$
$$\lambda_{E(q_0),E(q_1)}(0)=E(q_1),\ \lambda'_{E(q_0),E(q_1)}(0)=\alpha(E(q_0),E(q_1))\eta(E(q_0),E(q_1)),\ \lambda_{E(q_0),E(q_1)}(s)\in S.$$
Hence, the curve $\Gamma=(\gamma_{E(p_0),E(p_1)},\lambda_{E(q_0),E(q_1)}):[0,1]\rightarrow M\times M$ is a geodesic in $S\times S$, with
\begin{eqnarray*}
&&\Gamma(0)=(E(p_1),E(q_1)\text{ and }\\
&&\Gamma'(0)=(\alpha(E(p_0),E(p_1))\eta(E(p_0),E(p_1)),\alpha(E(q_0),E(q_1))\eta(E(q_0),E(q_1))).
\end{eqnarray*}
Since $\Gamma$ is a geodesic in $S\times S$, 
$$\Gamma(s)=(\gamma_{E(p_0),E(p_1)}(s),\lambda_{E(q_0),E(q_1)}(s))\in S\times S\ \forall s\in[0,1].$$
Now from the definition of $f$ and the geodesic $(\alpha,E)$-preinvexity of $F$, we get
\begin{eqnarray*}
f(\gamma_{E(p_0),E(p_1)}(s))&=& \inf_{q\in S}F(\gamma_{E(p_0),E(p_1)}(s),E(q))\\
&\leq & F(\gamma_{E(p_0),E(p_1)}(s),\lambda_{E(q_0),E(q_1)}(s))\\
&\leq & sF(E(p_0),E(q_0))+(1-s)F(E(p_1),E(q_1))\\
&\leq & s(f(E(p_0))+\epsilon)+(1-s)(f(E(p_1))+\epsilon)\\
&=& sf(E(p_0))+(1-s)f(E(p_1))+\epsilon.
\end{eqnarray*}
Since $\epsilon$ is arbitrary number, therefore
$$f(\gamma_{E(p_0),E(p_1)}(s))\leq sf(E(p_0))+(1-s)f(E(p_1)).$$
Hence, we get our theorem.
\end{proof}

\end{document}